\documentclass[12pt]{article}
\usepackage{amsmath,amsthm,amssymb}
\usepackage{amssymb,latexsym}

\newtheorem{theorem}{Theorem}

\newtheorem{lemma}{Lemma}

\begin{document}
\baselineskip=17pt

\title{\bf Barban--Davenport--Halberstam type theorems for exponential sums and Piatetski-Shapiro primes}

\author{\bf S. I. Dimitrov}

\date{2024}

\maketitle

\begin{abstract}
In this paper we establish three Barban--Davenport--Halberstam type theorems.  
Namely for exponential sums over primes, for Piatetski-Shapiro primes and for exponential sums over Piatetski-Shapiro primes.
\quad\\
{\bf  2020 Math.\ Subject Classification}: 11L20 $\cdot$ 11L40 $\cdot$ 11N05 $\cdot$ 11N36
\end{abstract}

\section{Introduction and statement of the result}
\indent

The Barban--Davenport--Halberstam theorem concerns the distribution of primes in arithmetic progression. It states that for any fixed $A>0$ we have
\begin{equation*}
\sum\limits_{q\le Q}\sum\limits_{a=1\atop{(a,\,q)=1}}^q\Bigg|\sum_{n\le X\atop{n\equiv a\,( \textmd{mod}\, q)}} \Lambda(n)-\frac{X}{\varphi(q)}\Bigg|^2\ll XQ\log X\,,
\end{equation*}
where $X(\log X)^{-A}\leq Q\leq X$, $\varphi (n)$ is Euler's function and $\Lambda(n)$ is von Mangoldt's function. See (\cite{Davenport}, Ch. 29) for a proof of this theorem.
Afterwards Montgomery \cite{Montgomery} and Hooley \cite{Hooley} each gave asymptotic formulas of this result, valid for various ranges of $Q$. \
Let $\mathbb{P}$ denotes the set of all prime numbers.
In 1953 Piatetski-Shapiro \cite{Shapiro} has shown that for any fixed $\frac{11}{12}<\gamma< 1$ the set
\begin{equation*}
\mathbb{P}_\gamma=\{p\in\mathbb{P}\;\;|\;\; p= [n^{1/\gamma}]\;\; \mbox{ for some } n\in \mathbb{N}\}
\end{equation*}
is infinite.
The prime numbers of the form $p = [n^{1/\gamma}]$ are called Piatetski-Shapiro primes of type $\gamma$.
His result states that
\begin{equation}\label{Shapiroformula}
\sum\limits_{p\leq X\atop{p=[n^{1/\gamma}]}}1=\frac{X^\gamma}{\log X}+\mathcal{O}\left(\frac{X^\gamma}{\log^2X}\right)
\end{equation}
for $\frac{11}{12}<\gamma<1$.
The best results up to now belong to Rivat and Sargos \cite{Rivat-Sargos} with \eqref{Shapiroformula} for $\frac{2426}{2817}<\gamma<1$ and to Rivat and Wu \cite{Rivat-Wu} with
\begin{equation*}
\sum\limits_{p\leq X\atop{p=[n^{1/\gamma}]}}1\gg\frac{X^\gamma}{\log X}
\end{equation*}
for $\frac{205}{243}<\gamma<1$.
In this connection Peneva \cite{Peneva}, Wang and Cai \cite{Cai}, Lu \cite{Lu} and J. Li, M. Zhang and F. Xue \cite{Zhang-Li} proved a Bombieri -- Vinogradov type theorems for Piatetski-Shapiro primes.
In \cite{Dimitrov1} and \cite{Dimitrov2} the author established a Bombieri -- Vinogradov type result for exponential sums over primes and for exponential sums over Piatetski-Shapiro primes, respectively.
As a continuation of these studies first we establish the Barban--Davenport--Halberstam theorem for exponential sums over primes.
\begin{theorem}\label{Theorem1}
Let $1<c<3$, $c\neq2$, $0<\mu<1$, $|t|\leq X^{\frac{2}{3}-c-\delta}$ for a sufficiently small $\delta>0$ and $A>0$ be fixed. Then 
\begin{equation*}
\sum\limits_{q\le Q}\sum\limits_{a=1\atop{(a,\,q)=1}}^q\Bigg|\sum_{\mu X<n\le X\atop{n\equiv a\,( \textmd{mod}\, q)}} \Lambda(n) e(t n^c)-\frac{1}{\varphi(q)}\int\limits_{\mu X}^{X}e(t y^c)\,dy\Bigg|^2\ll XQ\log X\,,
\end{equation*}
where $X(\log X)^{-A}\leq Q\leq X$. Here $e(t)=e^{2\pi it}$.
\end{theorem}
Assuming that the Generalized Riemann Hypothesis (GRH) is true we prove the Barban--Davenport--Halberstam theorem over Piatetski-Shapiro primes.
\begin{theorem}\label{Theorem2}
Assume the GRH. Let $\frac{2426}{2817}<\gamma<1$ and $0<\mu<1$. Then 
\begin{equation*}
\sum\limits_{q\le Q}\sum\limits_{a=1\atop{(a,\,q)=1}}^q\Bigg|\sum\limits_{\mu X<n\le X\atop{n\equiv a\, ( \textmd{mod}\, q)\atop{n=[k^{1/\gamma}]}}}\Lambda(n)-\frac{X^\gamma}{\varphi(q)}\Bigg|^2\ll X^\gamma Q\log X\,,
\end{equation*}
where $X^\gamma(\log X)^{-2}\leq Q\leq X^\gamma$.
\end{theorem}
The third result concerns the Barban--Davenport--Halberstam theorem for exponential sums over Piatetski-Shapiro primes.
\begin{theorem}\label{Theorem3}
Assume the GRH. Let $0<\mu<1$, $\frac{11}{12}<\gamma<1<c<3$, $c\neq2$ and $|t|\leq X^{\frac{4\gamma-3c-1}{3}-\delta}$ for a sufficiently small $\delta>0$, and $A>0$ be fixed. Then
\begin{equation*}
\sum\limits_{q\le Q}\sum\limits_{a=1\atop{(a,\,q)=1}}^q
\Bigg|\sum\limits_{\mu X<n\le X\atop{n\equiv a\, ( \textmd{mod}\, q)\atop{n=[k^{1/\gamma}]}}}\Lambda(n)n^{1-\gamma}e(t n^c)-\frac{\gamma}{\varphi(d)}\int\limits_{\mu X}^Xe(t y^c)\,dy\Bigg|^2\ll X^{2-\gamma} Q\log X\,,
\end{equation*}
where $X^\gamma(\log X)^{-A}\leq Q\leq X^\gamma$.
\end{theorem}

\section{Notations}
\indent

Assume that $X$ is a sufficiently large positive number. The letter $p$ will always denote prime number.
We use $[t]$ and $\{t\}$ to denote the integer part, respectively, the fractional part of $t$. Moreover $e(y)=e^{2\pi i y}$ and $\psi(t)=\{t\}-1/2$. 
Instead of $m\equiv n\,\pmod {d}$ we write for simplicity $m\equiv n\,(d)$.
The letter $\chi$ denotes a Dirichlet character to a given modulus.
The sums $\sum_{\chi(d)}$ and $\sum_{\chi(d)}^*$ denotes respectively summation over all characters and all primitive characters modulo $d$.
Throughout this paper we suppose that $0<\mu<1$.
Define
\begin{align}
\label{Psichi}
&\Psi(X,\chi)=\sum\limits_{n\leq X}\Lambda(n)\chi(n)\,,\\
\label{Psichit}
&\Psi(X,\chi,t)=\sum\limits_{\mu X<n\leq X}\Lambda(n)\chi(n)e(t n^c)\,.
\end{align}
We use $\rho=\alpha+i\beta$ to denote a non-trivial zero of the Dirichlet L-function $L(s,\chi)$ and we write
\begin{equation}\label{Irhoht}
I_\rho(h,t)=\int\limits_{\mu X}^{X_1}y^{\rho-1}e\big(ty^c-h y^\gamma\big)\,dy
\end{equation}
for some $X_1\in (\mu X, X]$. We let $\sum_{|\beta|\leq T}$ denote summation over all $\rho$ with $|\beta|\leq T$.
If $T\geq2$ and $0\leq\sigma\leq1$ we denote by $N(\sigma,T,\chi)$ the number of zeros
$\rho=\alpha+i\beta$ of the function $L(s,\chi)$ counted with multiplicity in the rectangle $\sigma\leq\alpha\leq1$,\; $ |\beta|\leq T$.

\section{Lemmas}
\indent

\begin{lemma}\label{Vaaler}
For every $H\geq1$, we have
\begin{equation*}
\psi(t)=\sum\limits_{1\leq|h|\leq H}a(h)e(ht)+\mathcal{O}\Bigg(\sum\limits_{|h|\leq H}b(h)e(ht)\Bigg)\,,
\end{equation*}
where
\begin{equation}\label{bh}
a(h)\ll\frac{1}{|h|}\,,\quad b(h)\ll\frac{1}{H}\,.
\end{equation}
\end{lemma}
\begin{proof}
See \cite{Vaaler}.
\end{proof}

\begin{lemma}\label{Exponentpairs}
Let $|f^{(m)}(u)|\asymp YX^{1-m}$  for $1\leq X<u<X_0\leq2X$ and $m\geq1$.
Then
\begin{equation*}
\bigg|\sum_{X<n\le X_0}e(f(n))\bigg|
\ll Y^\varkappa X^\lambda +Y^{-1}\,,
\end{equation*}
where $(\varkappa, \lambda)$ is any exponent pair.
\end{lemma}
\begin{proof}
See (\cite{Graham-Kolesnik}, Ch. 3).
\end{proof}

\begin{lemma}\label{SIasympt} Let $1<c<3$, $c\neq2$, $0<\mu<1$ and $|t|\leq X^{1-c-\varepsilon}$.
Then 
\begin{equation*}
\sum\limits_{\mu X<p\leq X} e(t p^c)\log p=\int\limits_{\mu X}^{X}e(t y^c)\,dy+\mathcal{O}\left(\frac{X}{e^{(\log X)^{1/5}}}\right)\,.
\end{equation*}
\end{lemma}
\begin{proof}
See (\cite{Tolev1}, Lemma 14).
\end{proof}

\begin{lemma}\label{Irhohtest} Let $0<\mu<1$, $0<\gamma<1<c$ and $\mu X\leq T\leq X$. Then 
\begin{equation*}
I_\rho(h,t)\ll\begin{cases}\frac{X^\alpha}{\sqrt{|t| X^c+|h|X^\gamma}}\quad\mbox{ for }\quad|\beta|<4\pi c|t| X^c+4\pi\gamma|h| X^\gamma\,,\\
\frac{X^\alpha}{|\beta|}\hspace{19.4mm}\mbox{ for }\quad 4\pi c|t| X^c+4\pi\gamma|h| X^\gamma\leq|\beta|\leq T\,.
\end{cases}
\end{equation*}
Here $I_\rho(h,t)$ is defined by \eqref{Irhoht}.
\end{lemma}
\begin{proof}
Inspecting the arguments presented in (\cite{Tolev2}, Lemma 10), the reader will readily see that the proof of Lemma \ref{Irhohtest}  can be obtained in the same way.
\end{proof}

\begin{lemma}\label{Kumchev} Let $\frac{11}{12}<\gamma<1<c$ and $|t|\leq X^{\gamma-c+\delta}$ for a sufficiently small $\delta>0$. 
Then 
\begin{equation*}
\sum\limits_{\mu X<n\leq X}\Lambda(n)\big(\psi(-(n+1)^\gamma)-\psi(-n^\gamma)\big)e(t n^c)\ll X^{\frac{11}{12}+\delta}\,.
\end{equation*}
\end{lemma}
\begin{proof}
See (\cite{Kumchev}, Lemma 10).
\end{proof}

\begin{lemma}\label{SumPsixchit1}
Let $\delta$, $\xi$, $\mu$ and $c$ be positive real numbers, such that
\begin{equation*}
\xi+7\delta<2\,,\quad 3\xi+6\delta<2\,,\quad  0<\mu<1\,,\quad  c>1\,.
\end{equation*}
Let $Q=X^\delta$ and $D\geq2$. Then for $X^{-c}(\log X)^D\leq|t|\leq X^{\xi-c}$ the inequality
\begin{equation*}
\sum\limits_{1<q\le Q}\frac{1}{\varphi(q)}\sideset{}{^*}\sum\limits_{\chi(q)}\Bigg|\sum\limits_{\mu X<n\leq X}\Lambda(n)\chi(n)e(t n^c)\Bigg|\ll \frac{X}{(\log X)^{\frac{D}{2}-17}}
\end{equation*}
holds. 
\end{lemma}
\begin{proof}
See (\cite{Li}, Lemma 2.8, \cite{Tolev2}, Lemma 10 and \cite{Zhu}, Lemma 4.5).
\end{proof}

\begin{lemma}\label{largesieve} (Large Sieve)
For any complex numbers $a_n$ and positive integers  $M, N, Q$ we have
\begin{equation*}
\sum\limits_{q\leq Q}\frac{q}{\varphi(q)}\sideset{}{^*}\sum\limits_{\chi(q)}\bigg|\sum\limits_{n=M+1}^{M+N}a_n\chi(n)\bigg|^2
\ll\big(N + Q^2\big)\sum\limits_{n=M+1}^{M+N}|a_n|^2
\end{equation*}
\end{lemma}
\begin{proof}
See (\cite{Iwaniec-Kowalski}, Theorem 7.13).
\end{proof}

\begin{lemma}\label{Psixchiest}
Let $q\leq\log^AX$. Then
\begin{equation*}
\Psi(X,\chi)\ll\frac{X}{e^{c(\log X)^{1/2}}}
\end{equation*}
for any nonprincipal character $\chi$ modulo $q$. Here $\Psi(X,\chi)$ is defined by \eqref{Psichi}.
\end{lemma}
\begin{proof} See (\cite{Davenport}, p. 132).
\end{proof}

\begin{lemma}\label{PsiXchiformula}
Let $\chi$ is primitive character to modulus $r>1$. If $1\leq T\leq X$ then
\begin{equation*}
\Psi(X,\chi)=-\sum\limits_{|\beta|\leq T}\frac{X^\rho}{\rho}+\sum\limits_{|\beta|<1}\frac{1}{\rho}+\mathcal{O}\bigg(\frac{X\log^2(rX)}{T}\bigg)\,.
\end{equation*}
Here $\Psi(X,\chi)$ is defined by \eqref{Psichi}.
\end{lemma}
\begin{proof} See (\cite{Davenport}, Ch. 19).
\end{proof}

\begin{lemma}\label{Montgomery}
If $Q\geq1$ and $T\geq2$ then
\begin{equation*}
\sum\limits_{q\le Q}\sideset{}{^*}\sum\limits_{\chi(q)}
N(\sigma,T,\chi)\ll\begin{cases}(Q^2T)^{\frac{3(1-\sigma)}{2-\sigma}}(\log QT)^9
\quad\mbox{if}\quad \frac{1}{2}\leq\sigma\leq\frac{4}{5} \,,\\
(Q^2T)^{\frac{2(1-\sigma)}{\sigma}}(\log QT)^{14}\quad\mbox{if}\quad \frac{4}{5}\leq\sigma\leq1\,.
\end{cases}
\end{equation*}
\end{lemma}
\begin{proof} See (\cite{Montgomery}, Theorem 12.2).
\end{proof}

\begin{lemma}\label{GRH}
Assume the GRH. Let $Q\geq 1$, $T\geq2$ and
\begin{equation}\label{LTQ}
L(T, Q)=\sum\limits_{q\le Q}\sideset{}{^*}\sum\limits_{\chi(q)}\sum\limits_{|\beta|\leq T}1\,.
\end{equation}
Then
\begin{equation*}
L(T, Q)\ll Q^2 T (\log QT)^9\,.
\end{equation*}
\end{lemma}
\begin{proof}
From \eqref{LTQ} and Lemma \ref{Montgomery} it follows
\begin{equation*}
L(T, Q)=\sum\limits_{q\le Q}\sideset{}{^*}\sum\limits_{\chi(q)}N\big(1/2,T,\chi\big)\ll Q^2 T (\log QT)^9\,.
\end{equation*}
\end{proof}

\begin{lemma}\label{SumPsixchit2}
Let  $0<\mu<1$, $c>1$, $B>0$, $C>0$ and $|t|\leq X^{\frac{2}{3}-c-\delta}$ for a sufficiently small $\delta>0$. Then
\begin{equation*}
\sum\limits_{1<q\le \log^CX}\frac{1}{\varphi(q)}\sideset{}{^*}\sum\limits_{\chi(q)}\Bigg|\sum\limits_{\mu X<n\leq X}\Lambda(n)\chi(n)e(t n^c)\Bigg|\ll \frac{X}{\log^BX}\,.
\end{equation*}
\end{lemma}
\begin{proof}
We consider two cases.

\textbf{Case 1} 

\begin{equation}\label{tXcB}
|t|\leq X^{-c}(\log X)^{2B+34}\,.
\end{equation}
From \eqref{tXcB}, Abel's summation formula and Lemma \ref{Psixchiest} we get
\begin{align*}
&\sum\limits_{1<q\le \log^CX}\frac{1}{\varphi(q)}\sideset{}{^*}\sum\limits_{\chi(q)}\Bigg|\sum\limits_{\mu X<n\leq X}\Lambda(n)\chi(n)e(t n^c)\Bigg|\nonumber\\
&=\sum\limits_{1<q\le \log^CX}\frac{1}{\varphi(q)}\sideset{}{^*}\sum\limits_{\chi(q)}\Bigg|e(tX^c)\sum\limits_{\mu X<n\leq X}\Lambda(n)\chi(n)\nonumber\\
&\hspace{40mm}-\int\limits_{\mu X}^X\left(\sum\limits_{\mu X<n\leq y}\Lambda(n)\chi(n)\right)\frac{d}{dy}e(t y^c)\,dy\Bigg|\nonumber\\
&\ll\big(1+|t|X^c\big)\frac{X}{e^{(\log X)^{1/3}}}\ll\frac{X}{e^{(\log X)^{1/4}}}\,.
\end{align*}

\textbf{Case 2} 

\begin{equation*}
X^{-c}(\log X)^{2B+34}\leq|t|\leq X^{\frac{2}{3}-c-\delta}
\end{equation*}

This case follows from Lemma \ref{SumPsixchit1}. The lemma is proved.
 
\end{proof}

\begin{lemma}\label{SumLchigama}
Assume the GRH. Let $0<\mu<1$, $\frac{2}{5}<\gamma<1$, $B>0$ and $C>0$. Then
\begin{equation*}
\sum\limits_{1<q\le \log^CX}\frac{1}{\varphi(q)}\sideset{}{^*}\sum\limits_{\chi(q)}\Bigg|\sum\limits_{\mu X<n\leq X\atop{n=[k^{1/\gamma}]}}\Lambda(n)\chi(n)\Bigg|\ll \frac{X^\gamma}{\log^BX}\,.
\end{equation*}
\end{lemma}
\begin{proof} We have 
\begin{align}\label{SumLchigamaest1}
&\sum\limits_{1<q\le \log^CX}\frac{1}{\varphi(q)}\sideset{}{^*}\sum\limits_{\chi(q)}\Bigg|\sum\limits_{\mu X<n\leq X\atop{n=[k^{1/\gamma}]}}\Lambda(n)\chi(n)\Bigg|\nonumber\\
&=\sum\limits_{1<q\le \log^CX}\frac{1}{\varphi(q)}\sideset{}{^*}\sum\limits_{\chi(q)}\Bigg|\sum\limits_{\mu X<n\leq X}\Lambda(n)\chi(n)\big([-n^\gamma]-[-(n+1)^\gamma]\big)\Bigg|\ll S_1+S_2\,,
\end{align}
where
\begin{align}
\label{Schi1}
&S_1=\sum\limits_{1<q\le \log^CX}\frac{1}{\varphi(q)}\sideset{}{^*}\sum\limits_{\chi(q)}\Bigg|\sum\limits_{\mu X<n\leq X}\Lambda(n)\chi(n)\big((n+1)^\gamma-n^\gamma\big)\Bigg|\,,\\
\label{Schi2}
&S_2=\sum\limits_{1<q\le \log^CX}\frac{1}{\varphi(q)}\sideset{}{^*}\sum\limits_{\chi(q)}\Bigg|\sum\limits_{\mu X<n\leq X}\Lambda(n)\chi(n)\big(\psi(-(n+1)^\gamma)-\psi(-n^\gamma)\big)\Bigg|\,.
\end{align}
Using \eqref{Schi1}, Abel's summation formula and Lemma \ref{Psixchiest} we derive
\begin{align}\label{S1est1}
S_1&=\sum\limits_{1<q\le \log^CX}\frac{1}{\varphi(q)}\sideset{}{^*}\sum\limits_{\chi(q)}\Bigg|\sum\limits_{\mu X<n\leq X}\Lambda(n)\chi(n)\Big(\gamma n^{\gamma-1}+\mathcal{O}\left(n^{\gamma-2}\right)\Big)\Bigg|\nonumber\\
&=\sum\limits_{1<q\le \log^CX}\frac{1}{\varphi(q)}\sideset{}{^*}\sum\limits_{\chi(q)}\Bigg|\gamma X^{\gamma-1}\sum\limits_{\mu X<n\leq X}\Lambda(n)\chi(n)\nonumber\\
&\hspace{40mm}+\gamma(1-\gamma)\int\limits_{\mu X}^X\left(\sum\limits_{\mu X<n\leq y}\Lambda(n)\chi(n)\right)y^{\gamma-2}\,dy+\mathcal{O}(1)\Bigg|\nonumber\\
&\ll\frac{X^\gamma}{e^{(\log X)^{1/3}}}\,.
\end{align}
Next we consider $S_2$. From \eqref{Schi2} and Lemma \ref{Vaaler} we obtain
\begin{align}\label{S2345}
S_2=S_3+S_4+S_5\,,
\end{align}
where
\begin{align}
\label{S3}
&S_3=\sum\limits_{1<q\le \log^CX}\frac{1}{\varphi(q)}\sideset{}{^*}\sum\limits_{\chi(q)}\Bigg|\sum\limits_{\mu X<n\leq X}\Lambda(n)\chi(n)\sum\limits_{1\leq|h|\leq H}a(h)\Big(e(-h(n+1)^\gamma)-e(-hn^\gamma)\Big)\Bigg|\,,\\
\label{S4}
&S_4\ll\sum\limits_{1<q\le \log^CX}\frac{1}{\varphi(q)}\sideset{}{^*}\sum\limits_{\chi(q)}\Bigg|\sum\limits_{\mu X<n\leq X}\Lambda(n)\chi(n)\sum\limits_{|h|\leq H}b(h)e(-hn^\gamma)\Bigg|\,,\\
\label{S5}
&S_5\ll\sum\limits_{1<q\le \log^CX}\frac{1}{\varphi(q)}\sideset{}{^*}\sum\limits_{\chi(q)}\Bigg|\sum\limits_{\mu X<n\leq X}\Lambda(n)\chi(n)\sum\limits_{|h|\leq H}b(h)e(-h(n+1)^\gamma)\Bigg|\,.
\end{align}
By \eqref{bh} and \eqref{S3} we deduce
\begin{equation}\label{S3est1}
S_3\ll\sum\limits_{1<q\le \log^CX}\frac{1}{\varphi(q)}\sideset{}{^*}\sum\limits_{\chi(q)}\sum\limits_{1\leq|h|\leq H}\frac{1}{|h|}\Bigg|\sum\limits_{\mu X<n\leq X}\Lambda(n)\chi(n)\Phi_h(n)e(-hn^\gamma)\Bigg|\,,
\end{equation}
where
\begin{equation*}
\Phi_h(y)=e\big(hy^\gamma-h(y+1)^\gamma\big)-1\,.
\end{equation*}
Bearing in mind the estimates
\begin{equation*}
\Phi_h(y)\ll|h|y^{\gamma-1}\,,  \quad \Phi'_h(y)\ll|h|y^{\gamma-2}
\end{equation*}
and using Abel's summation formula from \eqref{S3est1} we get
\begin{align}\label{S3est2}
S_3&\ll\sum\limits_{1<q\le \log^CX}\frac{1}{\varphi(q)}\sideset{}{^*}\sum\limits_{\chi(q)}\left(\sum\limits_{1\leq |h|\leq H}\frac{1}{|h|}\Bigg|\Phi_h(X)\sum\limits_{\mu X<n\leq X}\Lambda(n)\chi(n)e(-hn^\gamma)\Bigg|\right.\nonumber\\
&\left. \hspace{40mm}+\sum\limits_{1\leq |h|\leq H}\frac{1}{|h|}\int\limits_{\mu X}^X\Bigg|\Phi'_h(y)\sum\limits_{\mu X<n\leq y}\Lambda(n)\chi(n)e(-hn^\gamma)\Bigg|\,dy\right)\nonumber\\
&\ll X^{\gamma-1}\sum\limits_{1<q\le \log^CX}\frac{1}{\varphi(q)}\sideset{}{^*}\sum\limits_{\chi(q)}\sum\limits_{1\leq |h|\leq H}\Bigg|\sum\limits_{\mu X<n\leq X_1}\Lambda(n)\chi(n)e(-hn^\gamma)\Bigg|\,,
\end{align}
for some $X_1\in (\mu X, X]$.
Applying Abel's summation formula and Lemma \ref{PsiXchiformula} we find
\begin{align}\label{SumLambda}
&\sum\limits_{\mu X<n\leq X_1}\Lambda(n)\chi(n)e(-hn^\gamma)\nonumber\\
&=e(-h X_1^\gamma)\sum\limits_{\mu X<n\leq X_1}\Lambda(n)\chi(n)-\int\limits_{\mu X}^{X_1}\left(\sum\limits_{\mu X<n\leq t}\Lambda(n)\chi(n)\right)\frac{d}{dt}e(-h t^\gamma)\,dt\nonumber\\
&=\big(\Psi(X_1,\chi)-\Psi(\mu X,\chi)\big)e(-h X_1^\gamma)-\int\limits_{\mu X}^{X_1}\big(\Psi(t,\chi)-\Psi(\mu X,\chi)\big)\frac{d}{dt}e(-h t^\gamma)\,dt\nonumber\\
&=\left[-\sum\limits_{|\beta|\leq T}\frac{X_1^\rho-(\mu X)^\rho}{\rho}+\mathcal{O}\bigg(\frac{X\log^2X}{T}\bigg)\right]e(-h X_1^\gamma)\nonumber\\
&+\int\limits_{\mu X}^{X_1}\left[\sum\limits_{|\beta|\leq T}\frac{t^\rho-(\mu X)^\rho}{\rho}+\mathcal{O}\bigg(\frac{X\log^2X}{T}\bigg)\right]\frac{d}{dt}e(-h t^\gamma)dt\nonumber\\
&=-\sum\limits_{|\beta|\leq T}\int\limits_{\mu X}^{X_1}t^{\rho-1}e(-h t^\gamma)dt+\mathcal{O}\bigg(\frac{X\log^2X}{T}\bigg)+\mathcal{O}\bigg(|h|\frac{X^{1+\gamma}\log^2X}{T}\bigg)\nonumber\\
&=-\sum\limits_{|\beta|\leq T}I_\rho(h,0)+\mathcal{O}\bigg(|h|\frac{X^{1+\gamma}\log^2X}{T}\bigg)\,,
\end{align}
where $I_\rho(h,t)$ is defined by \eqref{Irhoht}. Now \eqref{S3est2}  and \eqref{SumLambda} give us 
\begin{equation}\label{S3est3}
S_3\ll X^{\gamma-1}\sum\limits_{1\leq |h|\leq H}\sum\limits_{q\le \log^CX}\frac{1}{\varphi(q)}\sideset{}{^*}\sum\limits_{\chi(q)}\sum\limits_{|\beta|\leq T}\big|I_\rho(h,0)\big|+T^{-1}H^2X^{2\gamma}(\log X)^{C+2}\,.
\end{equation}
Put
\begin{equation}\label{H}
H=\sqrt[3]{\frac{X^{1-\gamma}}{(\log X)^{2B+2C+22}}}\,,
\end{equation}
\begin{equation}\label{T}
T=H^2X^\gamma(\log X)^{B+C+2}\,.
\end{equation}
From \eqref{H} and \eqref{T} it is easy to see that $2\leq T\leq \mu X$. Taking into account \eqref{S3est3}, \eqref{T} and Lemma \ref{Irhohtest} we derive
\begin{equation}\label{S3est4}
S_3\ll X^{\gamma-1}(S'_3+S''_3)+\frac{X^\gamma}{\log^BX}\,,
\end{equation}
where $S'_3$ is the contribution of the summands for which $|\beta|<4\pi\gamma|h| X^\gamma$ and $S''_3$ is the contribution of the summands such that $4\pi\gamma|h| X^\gamma\leq|\beta|\leq T$.
First we estimate $S'_3$. 
Using \eqref{LTQ}, \eqref{H}, Abel's summation formula and Lemma \ref{GRH} we obtain
\begin{align}\label{S'3est}
S'_3&\ll\frac{\log X}{\sqrt{X^\gamma}}\sum\limits_{1\leq |h|\leq H}\frac{1}{\sqrt{|h|}}\sum\limits_{q\le \log^CX}\frac{1}{q}\sideset{}{^*}\sum\limits_{\chi(q)}\sum\limits_{|\beta|\leq4\pi\gamma|h| X^\gamma}X^\frac{1}{2}\nonumber\\
&\ll(\log X)^2X^\frac{1-\gamma}{2}\sum\limits_{1\leq |h|\leq H}\frac{1}{\sqrt{|h|}} \max_{1\leq R \leq\log^CX } \big(R^{-1}L(4\pi\gamma|h| X^\gamma, R)\big)\nonumber\\
&\ll(\log X)^{11}X^\frac{1-\gamma}{2}\sum\limits_{1\leq |h|\leq H}\frac{1}{\sqrt{|h|}} \max_{1\leq R \leq\log^CX } \Big(R|h| X^\gamma\Big)\nonumber\\
&\ll X^\frac{1+\gamma}{2}H^\frac{3}{2}(\log X)^{C+11}\ll\frac{X}{\log^BX}\,.
\end{align}
Next we consider $S''_3$. From \eqref{LTQ}, \eqref{H}, Abel's summation formula and Lemma \ref{GRH} it follows
\begin{align}\label{S''3est}
S''_3&\ll(\log X)\sum\limits_{1\leq |h|\leq H}\sum\limits_{q\le \log^CX}\frac{1}{q}\sideset{}{^*}\sum\limits_{\chi(q)}\sum\limits_{4\pi\gamma|h| X^\gamma\leq|\beta|\leq T}\frac{X^\frac{1}{2}}{|\beta|}\nonumber\\
&\ll(\log X)^3X^\frac{1}{2}\sum\limits_{1\leq |h|\leq H} \max_{1\leq R \leq\log^CX } \max_{4\pi\gamma|h| X^\gamma\leq K\leq T}\big((RK)^{-1}L(K, R)\big)\nonumber\\
&\ll(\log X)^{12}X^\frac{1}{2}\sum\limits_{1\leq |h|\leq H} \max_{1\leq R \leq\log^CX } R\ll X^\frac{1}{2}H(\log X)^{C+12}\ll\frac{X}{\log^BX}\,.
\end{align}
Now  \eqref{S3est4}, \eqref{S'3est}  and \eqref{S''3est} yield 
\begin{equation}\label{S3est5}
S_3\ll \frac{X^\gamma}{\log^BX}\,.
\end{equation}
It remains to estimate the sums $S_4$ and $S_5$. By \eqref{bh} and \eqref{S4} we get
\begin{align}\label{S4est1}
S_4&\ll\sum\limits_{1<q\le \log^CX}\frac{1}{\varphi(q)}\sideset{}{^*}\sum\limits_{\chi(q)}\Bigg|b(0)\sum\limits_{\mu X<n\leq X}\Lambda(n)\chi(n)\nonumber\\
&\hspace{60mm}+\sum\limits_{\mu X<n\leq X}\Lambda(n)\chi(n)\sum\limits_{1\leq|h|\leq H}b(h)e(-hn^\gamma)\Bigg|\nonumber\\
&\ll H^{-1}S'_4+(\log X)^{C+1}S''_4\,,
\end{align}
where
\begin{align}
\label{S'4}
&S'_4=\sum\limits_{1<q\le \log^CX}\frac{1}{\varphi(q)}\sideset{}{^*}\sum\limits_{\chi(q)}\Bigg|\sum\limits_{\mu X<n\leq X}\Lambda(n)\chi(n)\Bigg|\,,\\
\label{S''4}
&S''_4=\sum\limits_{1\leq|h|\leq H}b(h)\sum\limits_{\mu X<n\leq X}e(-hn^\gamma)\,.
\end{align}
Using consistently \eqref{LTQ}, \eqref{T}, \eqref{S'4}, Lemma \ref{PsiXchiformula}, Abel's summation formula and Lemma \ref{GRH} we find
\begin{align}\label{S'4est}
&S'_4=\sum\limits_{1<q\le \log^CX}\frac{1}{\varphi(q)}\sideset{}{^*}\sum\limits_{\chi(q)}\big|\Psi(X,\chi)-\Psi(\mu X,\chi)\big|\nonumber\\
&=\sum\limits_{1<q\le \log^CX}\frac{1}{\varphi(q)}\sideset{}{^*}\sum\limits_{\chi(q)}\left|\sum\limits_{|\beta|\leq T}\frac{X^\rho-(\mu X)^\rho}{\rho}+\mathcal{O}\bigg(\frac{X\log^2X}{T}\bigg)\right|\nonumber\\
&\ll(\log X)\sum\limits_{1<q\le \log^CX}\frac{1}{q}\sideset{}{^*}\sum\limits_{\chi(q)}\sum\limits_{|\beta|\leq T}\frac{X^\frac{1}{2}}{|\beta|}+T^{-1}X(\log X)^{C+2}\nonumber\\
&\ll(\log X)^3X^\frac{1}{2} \max_{1\leq R \leq\log^CX } \max_{1\leq K\leq T}\big((RK)^{-1}L(K, R)\big)+T^{-1}X(\log X)^{C+2}\nonumber\\
&\ll(\log X)^{12}X^\frac{1}{2} \max_{1\leq R \leq\log^CX } R+T^{-1}X(\log X)^{C+2}\nonumber\\
&\ll X^\frac{1}{2}(\log X)^{C+12}+T^{-1}X(\log X)^{C+2}\,.
\end{align}
On the other hand \eqref{bh}, \eqref{H}, \eqref{S''4} and Lemma \ref{Exponentpairs} with exponent pair $\left(\frac{1}{2},\frac{1}{2}\right)$ imply
\begin{align}\label{S''4est}
S''_4&\ll(\log X)\sum\limits_{1\leq|h|\leq H}|b(h)|\Big(|h|^\frac{1}{2}X^\frac{\gamma}{2}+|h|^{-1}X^{1-\gamma}\Big)\nonumber\\
&\ll (\log X)H^{-1}\sum\limits_{1\leq|h|\leq H}\Big(|h|^\frac{1}{2}X^\frac{\gamma}{2}+|h|^{-1}X^{1-\gamma}\Big)\nonumber\\
&\ll \Big(H^\frac{1}{2}X^\frac{\gamma}{2}+H^{-1}X^{1-\gamma}\Big)\log^2X\ll \frac{X^\gamma}{(\log X)^{B+C+1}}\,.
\end{align}
Bearing in mind \eqref{H}, \eqref{T}, \eqref{S4est1}, \eqref{S'4est} and \eqref{S''4est} we obtain
\begin{equation}\label{S4est2}
S_4\ll  \frac{X^\gamma}{\log^BX}\,.
\end{equation}
In the same way for the sum defined by \eqref{S5} we get
\begin{equation}\label{S5est}
S_5\ll  \frac{X^\gamma}{\log^BX}\,.
\end{equation}
Summarizing \eqref{SumLchigamaest1}, \eqref{S1est1}, \eqref{S2345}, \eqref{S3est5}, \eqref{S4est2} and \eqref{S5est} we establish the statement it the lemma.
\end{proof}

\begin{lemma}\label{SumLexpchigama}
Assume the GRH. Let $0<\mu<1$, $\frac{2}{5}<\gamma<1<c<3$, $c\neq2$, $B>0$, $C>0$ and 
\begin{equation}\label{tlimits}
|t|\leq X^{\frac{4\gamma-3c-1}{3}-\delta}
\end{equation}
for a sufficiently small $\delta>0$. Then
\begin{equation*}
\sum\limits_{1<q\le \log^CX}\frac{1}{\varphi(q)}\sideset{}{^*}\sum\limits_{\chi(q)}\Bigg|\sum\limits_{\mu X<n\leq X\atop{n=[k^{1/\gamma}]}}\Lambda(n)\chi(n)e(t n^c)\Bigg|\ll \frac{X^\gamma}{\log^BX}\,.
\end{equation*}
\end{lemma}
\begin{proof} We have 
\begin{align}\label{SumLexpchigamaest1}
&\sum\limits_{1<q\le \log^CX}\frac{1}{\varphi(q)}\sideset{}{^*}\sum\limits_{\chi(q)}\Bigg|\sum\limits_{\mu X<n\leq X\atop{n=[k^{1/\gamma}]}}\Lambda(n)\chi(n)e(t n^c)\Bigg|\nonumber\\
&=\sum\limits_{1<q\le \log^CX}\frac{1}{\varphi(q)}\sideset{}{^*}\sum\limits_{\chi(q)}\Bigg|\sum\limits_{\mu X<n\leq X}\Lambda(n)\chi(n)\big([-n^\gamma]-[-(n+1)^\gamma]\big)e(t n^c)\Bigg|\nonumber\\
&\ll\Delta_1+\Delta_2\,,
\end{align}
where
\begin{align}
\label{Delta1}
&\Delta_1=\sum\limits_{1<q\le \log^CX}\frac{1}{\varphi(q)}\sideset{}{^*}\sum\limits_{\chi(q)}\Bigg|\sum\limits_{\mu X<n\leq X}\Lambda(n)\chi(n)\big((n+1)^\gamma-n^\gamma\big)e(t n^c)\Bigg|\,,\\
\label{Delta2}
&\Delta_2=\sum\limits_{1<q\le \log^CX}\frac{1}{\varphi(q)}\sideset{}{^*}\sum\limits_{\chi(q)}\Bigg|\sum\limits_{\mu X<n\leq X}\Lambda(n)\chi(n)\big(\psi(-(n+1)^\gamma)-\psi(-n^\gamma)\big)e(t n^c)\Bigg|\,.
\end{align}
By \eqref{Delta1} and Abel's summation formula we derive
\begin{align}\label{Delta1est1}
\Delta_1&=\sum\limits_{1<q\le \log^CX}\frac{1}{\varphi(q)}\sideset{}{^*}\sum\limits_{\chi(q)}\Bigg|\sum\limits_{\mu X<n\leq X}\Lambda(n)\chi(n)e(t n^c)\Big(\gamma n^{\gamma-1}+\mathcal{O}\left(n^{\gamma-2}\right)\Big)\Bigg|\nonumber\\
&=\sum\limits_{1<q\le \log^CX}\frac{1}{\varphi(q)}\sideset{}{^*}\sum\limits_{\chi(q)}\Bigg|\gamma X^{\gamma-1}\sum\limits_{\mu X<n\leq X}\Lambda(n)\chi(n)e(t n^c)\nonumber\\
&\hspace{40mm}+\gamma(1-\gamma)\int\limits_{\mu X}^X\left(\sum\limits_{\mu X<n\leq y}\Lambda(n)\chi(n)e(t n^c)\right)y^{\gamma-2}\,dy+\mathcal{O}(1)\Bigg|\nonumber\\
&\ll X^{\gamma-1}\sum\limits_{1<q\le \log^CX}\frac{1}{\varphi(q)}\sideset{}{^*}\sum\limits_{\chi(q)}\Bigg|\sum\limits_{\mu X<n\leq X_1}\Lambda(n)\chi(n)e(tn^c)\Bigg|+\log^CX\,,
\end{align}
for some $X_1\in (\mu X, X]$.
Using \eqref{Irhoht}, Lemma \ref{Irhohtest}, Lemma \ref{PsiXchiformula}, Lemma \ref{GRH} and arguing as in the estimation of the sum $S_3$ in Lemma \ref{SumLchigama} we obtain
\begin{align}\label{Delta1est2}
&\sum\limits_{1<q\le \log^CX}\frac{1}{\varphi(q)}\sideset{}{^*}\sum\limits_{\chi(q)}\Bigg|\sum\limits_{\mu X<n\leq X_1}\Lambda(n)\chi(n)e(tn^c)\Bigg|\nonumber\\
&\ll \sum\limits_{q\le \log^CX}\frac{1}{\varphi(q)}\sideset{}{^*}\sum\limits_{\chi(q)}\sum\limits_{|\beta|\leq T}\big|I_\rho(0,t)\big|+T^{-1}|t|X^{c+1}(\log X)^{C+2}\nonumber\\
&\ll |t|^\frac{1}{2}X^\frac{c+1}{2}(\log X)^{C+11}+X^\frac{1}{2}(\log X)^{C+12}+T^{-1}|t|X^{c+1}(\log X)^{C+2}\,.
\end{align}
Now \eqref{H}, \eqref{T}, \eqref{tlimits}, \eqref{Delta1est1} and \eqref{Delta1est2} lead to
\begin{equation}\label{Delta1est3}
\Delta_1\ll  \frac{X^\gamma}{\log^BX}\,.
\end{equation}
Next we estimate $\Delta_2$. From \eqref{Delta2} and Lemma \ref{Vaaler} we deduce
\begin{align}\label{Delta2345}
\Delta_2=\Delta_3+\Delta_4+\Delta_5\,,
\end{align}
where
\begin{align}
\label{Delta3}
&\Delta_3=\sum\limits_{1<q\le \log^CX}\frac{1}{\varphi(q)}\sideset{}{^*}\sum\limits_{\chi(q)}\Bigg|\sum\limits_{\mu X<n\leq X}\Lambda(n)\chi(n)e(t n^c)\nonumber\\
&\hspace{60mm}\times\sum\limits_{1\leq|h|\leq H}a(h)\Big(e(-h(n+1)^\gamma)-e(-hn^\gamma)\Big)\Bigg|\,,\\
\label{Delta4}
&\Delta_4\ll\sum\limits_{1<q\le \log^CX}\frac{1}{\varphi(q)}\sideset{}{^*}\sum\limits_{\chi(q)}\Bigg|\sum\limits_{\mu X<n\leq X}\Lambda(n)\chi(n)e(t n^c)\sum\limits_{|h|\leq H}b(h)e(-hn^\gamma)\Bigg|\,,
\end{align}

\begin{align}
\label{Delta5}
&\Delta_5\ll\sum\limits_{1<q\le \log^CX}\frac{1}{\varphi(q)}\sideset{}{^*}\sum\limits_{\chi(q)}\Bigg|\sum\limits_{\mu X<n\leq X}\Lambda(n)\chi(n)e(t n^c)\sum\limits_{|h|\leq H}b(h)e(-h(n+1)^\gamma)\Bigg|\,.
\end{align}
For the sum $\Delta_3$ denoted by \eqref{Delta3} we proceed as in the estimation of $S_3$ in Lemma \ref{SumLchigama} to see that
\begin{align}\label{Delta3est1}
\Delta_3&\ll X^{\gamma-1}\sum\limits_{1\leq |h|\leq H}\sum\limits_{q\le \log^CX}\frac{1}{\varphi(q)}\sideset{}{^*}\sum\limits_{\chi(q)}\sum\limits_{|\beta|\leq T}\big|I_\rho(h,t)\big|\nonumber\\
&+T^{-1}HX^\gamma\big(|t|X^c+HX^\gamma\big)(\log X)^{C+2}\,,
\end{align}
where $I_\rho(h,t)$ is defined by \eqref{Irhoht}. 
Let $H$ and $T$ be defined by \eqref{H} and \eqref{T} respectively. Bearing in mind \eqref{T}, \eqref{Delta3est1} and Lemma \ref{Irhohtest} we get
\begin{equation}\label{Delta3est2}
\Delta_3\ll X^{\gamma-1}(\Delta'_3+\Delta''_3)+\frac{X^\gamma}{\log^BX}\,,
\end{equation}
where $\Delta'_3$ is the contribution of the summands for which $\beta|<4\pi c|t| X^c+4\pi\gamma|h| X^\gamma$ and $\Delta''_3$ is the contribution of the summands such that $\beta|<4\pi c|t| X^c+4\pi\gamma|h| X^\gamma\leq|\beta|\leq T$.
First we estimate $\Delta'_3$. 
Using \eqref{LTQ}, \eqref{H}, \eqref{tlimits}, Abel's summation formula and Lemma \ref{GRH} we find
\begin{align}\label{Delta'3est}
\Delta'_3&\ll(\log X)\sum\limits_{1\leq |h|\leq H}\frac{1}{\sqrt{|t| X^c+|h|X^\gamma}}\sum\limits_{q\le \log^CX}\frac{1}{q}\sideset{}{^*}\sum\limits_{\chi(q)}\sum\limits_{|\beta|<4\pi c|t| X^c+4\pi\gamma|h| X^\gamma}X^\frac{1}{2}\nonumber\\
&\ll(\log X)^2X^\frac{1-\gamma}{2}\sum\limits_{1\leq |h|\leq H}\frac{1}{\sqrt{|h|}} \max_{1\leq R \leq\log^CX } \Big(R^{-1}L\big(4\pi c|t| X^c+4\pi\gamma|h| X^\gamma, R\big)\Big)\nonumber\\
&\ll(\log X)^{11}X^\frac{1-\gamma}{2}\sum\limits_{1\leq |h|\leq H}\frac{1}{\sqrt{|h|}} \max_{1\leq R \leq\log^CX } \Big(R\big(|t| X^c+|h| X^\gamma\big)\Big)\nonumber\\
&\ll X^\frac{1+\gamma}{2}H^\frac{3}{2}(\log X)^{C+11}\nonumber\\
&\ll\frac{X}{\log^BX}\,.
\end{align}
Next we consider $\Delta''_3$. Working as in the estimation of the sum $S''_3$ in Lemma \ref{SumLchigama} from \eqref{S''3est} we obtain
\begin{equation}\label{Delta''3est}
\Delta''_3\ll\frac{X}{\log^BX}\,.
\end{equation}
By \eqref{Delta3est2}, \eqref{Delta'3est}  and \eqref{Delta''3est} it follows
\begin{equation}\label{Delta3est3}
\Delta_3\ll \frac{X^\gamma}{\log^BX}\,.
\end{equation}
It remains to estimate the sums $\Delta_4$ and $\Delta_5$. By \eqref{bh} and \eqref{Delta4} we deduce
\begin{align}\label{Delta4est1}
&\Delta_4\ll\sum\limits_{1<q\le \log^CX}\frac{1}{\varphi(q)}\sideset{}{^*}\sum\limits_{\chi(q)}\Bigg|b(0)\sum\limits_{\mu X<n\leq X}\Lambda(n)\chi(n)e(t n^c)\nonumber\\
&\hspace{60mm}+\sum\limits_{\mu X<n\leq X}\Lambda(n)\chi(n)e(t n^c)\sum\limits_{1\leq|h|\leq H}b(h)e(-hn^\gamma)\Bigg|\,,\nonumber\\
&\ll H^{-1}\Delta'_4+(\log X)^{C+1}\Delta''_4\,,
\end{align}
where
\begin{align}
\label{Delta'4}
&\Delta'_4=\sum\limits_{1<q\le \log^CX}\frac{1}{\varphi(q)}\sideset{}{^*}\sum\limits_{\chi(q)}\Bigg|\sum\limits_{\mu X<n\leq X}\Lambda(n)\chi(n)e(t n^c)\Bigg|\,,\\
\label{Delta''4}
&\Delta''_4=\sum\limits_{1\leq|h|\leq H}b(h)\sum\limits_{\mu X<n\leq X}e(-hn^\gamma)\,.
\end{align}
Using  \eqref{Delta1est2} and \eqref{Delta'4} we find
\begin{equation}\label{Delta'4est}
\Delta'_4 \ll |t|^\frac{1}{2}X^\frac{c+1}{2}(\log X)^{C+11}+X^\frac{1}{2}(\log X)^{C+12}+T^{-1}|t|X^{c+1}(\log X)^{C+2}\,.
\end{equation}
On the other hand \eqref{S''4}, \eqref{S''4est} and \eqref{Delta''4} imply
\begin{equation}\label{Delta''4est}
\Delta''_4\ll \frac{X^\gamma}{(\log X)^{B+C+1}}\,.
\end{equation}
Taking into account \eqref{H}, \eqref{T}, \eqref{tlimits}, \eqref{Delta4est1}, \eqref{Delta'4est} and \eqref{Delta''4est} we get
\begin{equation}\label{Delta4est2}
\Delta_4\ll  \frac{X^\gamma}{\log^BX}\,.
\end{equation}
In the same way for the sum defined by \eqref{Delta5} we obtain
\begin{equation}\label{Delta5est}
\Delta_5\ll  \frac{X^\gamma}{\log^BX}\,.
\end{equation}
Summarizing \eqref{SumLexpchigamaest1}, \eqref{Delta1est3}, \eqref{Delta2345}, \eqref{Delta3est3}, \eqref{Delta4est2} and \eqref{Delta5est} we establish the statement it the lemma.
\end{proof}

\section{Proof of Theorem \ref{Theorem1}}
\indent

Define
\begin{equation}\label{deltachi}
\delta(\chi)=\begin{cases}1 \quad\mbox{ if $\chi$ is principal },\\
0\quad\mbox{ otherwise }.
\end{cases}
\end{equation}
By the orthogonality of characters we have
\begin{align}\label{orthogonality}
&\sum\limits_{\mu X<n\leq X\atop{n\equiv a\, (q)}}\Lambda(n)e(t n^c)-\frac{1}{\varphi(q)}\int\limits_{\mu X}^{X}e(t y^c)\,dy\nonumber\\
&=\sum\limits_{\mu X<n\leq X}\Lambda(n)e(t n^c)\frac{1}{\varphi(q)}\sum\limits_{\chi(q)}\chi(n)\overline{\chi}(a)-\frac{1}{\varphi(q)}\int\limits_{\mu X}^{X}e(t y^c)\,dy\nonumber\\
&=\frac{1}{\varphi(q)}\sum\limits_{\chi(q)}\left(\overline{\chi}(a)\sum\limits_{\mu X<n\leq X}\Lambda(n)\chi(n)e(t n^c)-\delta(\chi)\int\limits_{\mu X}^{X}e(t y^c)\,dy\right)\,.
\end{align}
Now \eqref{Psichit}, \eqref{orthogonality} and the formula 
\begin{equation}\label{sumchi12}
\sum\limits_{a=1\atop{(a,\,q)=1}}^q \chi_1(a)\overline{\chi}_2(a)
=\begin{cases}\varphi(q)\,,\;\mbox{  if  } \; \chi_1=\chi_2\,,\\
0\,,\hspace{7mm} \mbox{  if } \; \chi_1\neq\chi_2\
\end{cases}
\end{equation}
give us
\begin{align}\label{maxLambdaInt}
&\sum\limits_{a=1\atop{(a,\,q)=1}}^q\Bigg|\sum\limits_{\mu X<n\leq X\atop{n\equiv a\, (q)}}\Lambda(n)e(t n^c)-\frac{1}{\varphi(q)}\int\limits_{\mu X}^{X}e(t y^c)\,dy\Bigg|^2\nonumber\\
&=\frac{1}{\varphi^2(q)}\sum\limits_{a=1\atop{(a,\,q)=1}}^q\left|\sum\limits_{\chi(q)}\overline{\chi}(a)\left(\Psi(X,\chi,t)-\delta(\chi)\int\limits_{\mu X}^{X}e(t y^c)\,dy\right)\right|^2\nonumber\\
&=\frac{1}{\varphi(q)}\sum\limits_{\chi(q)}\left|\left(\Psi(X,\chi,t)-\delta(\chi)\int\limits_{\mu X}^{X}e(t y^c)\,dy\right)\right|^2\,.
\end{align}
Put
\begin{equation}\label{Sigma}
\Sigma=\sum\limits_{q\le Q}\sum\limits_{a=1\atop{(a,\,q)=1}}^q\Bigg|\sum\limits_{\mu X<n\leq X\atop{n\equiv a\, (q)}}\Lambda(n)e(t n^c)-\frac{1}{\varphi(q)}\int\limits_{\mu X}^{X}e(t y^c)\,dy\Bigg|^2\,.
\end{equation}
From \eqref{Psichit}, \eqref{deltachi}, \eqref{maxLambdaInt} and \eqref{Sigma} we  obtain
\begin{equation}\label{Sigmadecomp}
\Sigma=\Sigma_1+\Sigma_2\,,
\end{equation}
where
\begin{align}
\label{Sigma1}
&\Sigma_1=\sum\limits_{q\le Q}\frac{1}{\varphi(q)}\Bigg|\sum\limits_{ \mu X<n\leq X}\Lambda(n)e(t n^c)-\int\limits_{ \mu X}^{X}e(t y^c)\,dy+\mathcal{O}\Big(\log^2X\Big)\Bigg|^2\,,\\
\label{Sigma2}
&\Sigma_2=\sum\limits_{q\le Q}\frac{1}{\varphi(q)}\sum\limits_{\chi(q)\atop{\chi\neq\chi_0}}\big|\Psi(X,\chi,t)\big|^2\,.
\end{align}
First we estimate the sum $\Sigma_1$. Using \eqref{Sigma1} and Lemma \ref{SIasympt} we get
\begin{equation}\label{Sigma1est}
\Sigma_1\ll\frac{X^2}{e^{2(\log X)^{1/5}}}\sum\limits_{q\le Q}\frac{1}{\varphi(q)}\ll\frac{X^2}{\log^AX}\ll XQ\log X\,.
\end{equation}
Next we consider the sum $\Sigma_2$. Moving to primitive characters from \eqref{Sigma2} we deduce
\begin{align}\label{Sigma2est}
\Sigma_2&\ll\sum\limits_{q\le Q}\frac{1}{\varphi(q)}\sum\limits_{r|q\atop{r>1}}\sideset{}{^*}\sum\limits_{\chi(r)}\big|\Psi(X,\chi,t)\big|^2+Q\log^4X\nonumber\\
&\ll\sum\limits_{1<r\le Q}\left(\sum\limits_{q\le Q\atop{r|q}}\frac{1}{\varphi(q)}\right)\sideset{}{^*}\sum\limits_{\chi(r)}\big|\Psi(X,\chi,t)\big|^2+Q\log^4X\nonumber\\
&\ll\sum\limits_{1<r\le Q}\frac{1}{\varphi(r)}\left(\log\frac{Q}{r}\right)\sideset{}{^*}\sum\limits_{\chi(r)}\big|\Psi(X,\chi,t)\big|^2+Q\log^4X\nonumber\\
&=\Sigma_3+\Sigma_4+Q\log^4X\,,
\end{align}
where
\begin{align}
\label{Sigma3}
&\Sigma_3=\sum\limits_{1<q\le Q_0}\frac{1}{\varphi(q)}\left(\log\frac{Q}{q}\right)\sideset{}{^*}\sum\limits_{\chi(q)}\big|\Psi(X,\chi,t)\big|^2\,,\\
\label{Sigma4}
&\Sigma_4=\sum\limits_{Q_0<q\le Q}\frac{1}{\varphi(q)}\left(\log\frac{Q}{q}\right)\sideset{}{^*}\sum\limits_{\chi(q)}\big|\Psi(X,\chi,t)\big|^2\,,\\
\label{Q0}
&Q_0=(\log X)^{A+2}\,.
\end{align}
By \eqref{Psichit} and Chebyshev’s prime number theorem it follows
\begin{equation}\label{Psiest}
\Psi(X,\chi,t)\ll X\,.
\end{equation}
Now \eqref{Sigma3}, \eqref{Q0}, \eqref{Psiest} and Lemma \ref{SumPsixchit2} with $B=A+1$ yield
\begin{equation}\label{Sigma3est}
\Sigma_3\ll X(\log X)\sum\limits_{1<q\le Q_0}\frac{1}{\varphi(q)}\sideset{}{^*}\sum\limits_{\chi(q)}\big|\Psi(X,\chi,t)\big|\ll\frac{X^2}{\log^AX}\ll XQ\log X\,.
\end{equation}
Next we estimate $\Sigma_4$. Put
\begin{equation}\label{i0}
i_0=\left[\frac{\log\frac{Q}{Q_0}}{\log2}\right]\,.
\end{equation}
From \eqref{Psichit}, \eqref{Sigma4}, \eqref{Q0}, \eqref{i0} and Lemma \ref{largesieve} we derive 
\begin{align}\label{Sigma4est}
&\Sigma_4\ll\sum\limits_{1\leq i\le i_0}\sum\limits_{\frac{Q}{2^i}<q\leq\frac{Q}{2^{i-1}}}\frac{1}{\varphi(q)}\left(\log\frac{Q}{q}\right)\sideset{}{^*}\sum\limits_{\chi(q)}\big|\Psi(X,\chi,t)\big|^2\nonumber\\
&\ll Q^{-1}\sum\limits_{1\leq i\le i_0}2^i\log2^i\sum\limits_{\frac{Q}{2^i}<q\leq\frac{Q}{2^{i-1}}}\frac{q}{\varphi(q)}\sideset{}{^*}\sum\limits_{\chi(q)}\big|\Psi(X,\chi,t)\big|^2\nonumber\\
&\ll Q^{-1}\sum\limits_{1\leq i\le i_0}2^i i\left(X+\frac{Q^2}{4^i}\right)\sum\limits_{n\leq X}\Lambda^2(n)\nonumber\\      
&\ll Q^{-1}X(\log X)\sum\limits_{1\leq i\le i_0}2^i i\left(X+\frac{Q^2}{4^i}\right)\nonumber\\     
&\ll X^2Q_0^{-1}\log^3X +XQ\log X\nonumber\\     
&\ll XQ\log X\,.                     
\end{align}
Summarizing \eqref{Sigmadecomp}, \eqref{Sigma1est}, \eqref{Sigma2est}, \eqref{Sigma3est} and \eqref{Sigma4est} we obtain
\begin{equation}\label{Sigmaest}
\Sigma\ll XQ\log X\,. 
\end{equation}
Bearing in mind \eqref{Sigma} and \eqref{Sigmaest} we establish Theorem \ref{Theorem1}.

\section{Proof of Theorem \ref{Theorem2}}
\indent

Using \eqref{deltachi} and the orthogonality of characters we write
\begin{align}\label{orthogonality2}
\sum\limits_{\mu X<n\leq X\atop{n\equiv a\, (q)\atop{n=[k^{1/\gamma}]}}}\Lambda(n)-\frac{X^\gamma}{\varphi(q)}
&=\sum\limits_{\mu X<n\leq X\atop{n=[k^{1/\gamma}]}}\Lambda(n)\frac{1}{\varphi(q)}\sum\limits_{\chi(q)}\chi(n)\overline{\chi}(a)-\frac{X^\gamma}{\varphi(q)}\nonumber\\
&=\frac{1}{\varphi(q)}\sum\limits_{\chi(q)}\Bigg(\overline{\chi}(a)\sum\limits_{\mu X<n\leq X\atop{n=[k^{1/\gamma}]}}\Lambda(n)\chi(n)-\delta(\chi)X^\gamma\Bigg)\,.
\end{align}
Now  \eqref{sumchi12} and \eqref{orthogonality2} imply
\begin{align*}
&\sum\limits_{a=1\atop{(a,\,q)=1}}^q\Bigg|\sum\limits_{\mu X<n\leq X\atop{n\equiv a\, (q)\atop{n=[k^{1/\gamma}]}}}\Lambda(n)-\frac{X^\gamma}{\varphi(q)}\Bigg|^2\nonumber\\
\end{align*}

\begin{align}\label{maxLXgamma}
&=\frac{1}{\varphi^2(q)}\sum\limits_{a=1\atop{(a,\,q)=1}}^q\Bigg|\sum\limits_{\chi(q)}\overline{\chi}(a)\Bigg(\sum\limits_{\mu X<n\leq X\atop{n=[k^{1/\gamma}]}}\Lambda(n)\chi(n)-\delta(\chi)X^\gamma\Bigg)\Bigg|^2\nonumber\\
&=\frac{1}{\varphi(q)}\sum\limits_{\chi(q)}\Bigg|\Bigg(\sum\limits_{\mu X<n\leq X\atop{n=[k^{1/\gamma}]}}\Lambda(n)\chi(n)-\delta(\chi)X^\gamma\Bigg)\Bigg|^2\,.
\end{align}
Set
\begin{equation}\label{Gamma}
\Gamma=\sum\limits_{q\le Q}\sum\limits_{a=1\atop{(a,\,q)=1}}^q\Bigg|\sum\limits_{\mu X<n\leq X\atop{n\equiv a\, (q)\atop{n=[k^{1/\gamma}]}}}\Lambda(n)-\frac{X^\gamma}{\varphi(q)}\Bigg|^2\,.
\end{equation}
From \eqref{deltachi}, \eqref{maxLXgamma} and \eqref{Gamma} we deduce
\begin{equation}\label{Gammadecomp}
\Gamma=\Gamma_1+\Gamma_2\,,
\end{equation}
where
\begin{align}
\label{Gamma1}
&\Gamma_1=\sum\limits_{q\le Q}\frac{1}{\varphi(q)}\Bigg|\sum\limits_{\mu X<n\leq X\atop{n=[k^{1/\gamma}]}}\Lambda(n)-X^\gamma+\mathcal{O}\Big(\log^2X\Big)\Bigg|^2\,,\\
\label{Gamma2}
&\Gamma_2=\sum\limits_{q\le Q}\frac{1}{\varphi(q)}\sum\limits_{\chi(q)\atop{\chi\neq\chi_0}}\Bigg|\sum\limits_{\mu X<n\leq X\atop{n=[k^{1/\gamma}]}}\Lambda(n)\chi(n)\Bigg|^2\,.
\end{align}
By \eqref{Shapiroformula} and \eqref{Gamma1} we find
\begin{equation}\label{Gamma1est}
\Gamma_1\ll\frac{X^{2\gamma}}{\log^2X}\sum\limits_{q\le Q}\frac{1}{\varphi(q)}\ll\frac{X^{2\gamma}}{\log X}\ll X^\gamma Q\log X\,.
\end{equation}
Next we consider $\Gamma_2$. Moving to primitive characters from \eqref{Gamma2} we get
\begin{align}\label{Gamma2est}
\Gamma_2&\ll\sum\limits_{q\le Q}\frac{1}{\varphi(q)}\sum\limits_{r|q\atop{r>1}}\sideset{}{^*}\sum\limits_{\chi(r)}\Bigg|\sum\limits_{\mu X<n\leq X\atop{n=[k^{1/\gamma}]}}\Lambda(n)\chi(n)\Bigg|^2+Q\log^4X\nonumber\\
&\ll\sum\limits_{1<r\le Q}\left(\sum\limits_{q\le Q\atop{r|q}}\frac{1}{\varphi(q)}\right)\sideset{}{^*}\sum\limits_{\chi(r)}\Bigg|\sum\limits_{\mu X<n\leq X\atop{n=[k^{1/\gamma}]}}\Lambda(n)\chi(n)\Bigg|^2+Q\log^4X\nonumber\\
&\ll\sum\limits_{1<r\le Q}\frac{1}{\varphi(r)}\left(\log\frac{Q}{r}\right)\sideset{}{^*}\sum\limits_{\chi(r)}\Bigg|\sum\limits_{\mu X<n\leq X\atop{n=[k^{1/\gamma}]}}\Lambda(n)\chi(n)\Bigg|^2+Q\log^4X\nonumber\\
&=\Gamma_3+\Gamma_4+Q\log^4X\,,
\end{align}
where
\begin{align}
\label{Gamma3}
&\Gamma_3=\sum\limits_{1<q\le Q_1}\frac{1}{\varphi(q)}\left(\log\frac{Q}{q}\right)\sideset{}{^*}\sum\limits_{\chi(q)}\Bigg|\sum\limits_{\mu X<n\leq X\atop{n=[k^{1/\gamma}]}}\Lambda(n)\chi(n)\Bigg|^2\,,\\
\label{Gamma4}
&\Gamma_4=\sum\limits_{Q_1<q\le Q}\frac{1}{\varphi(q)}\left(\log\frac{Q}{q}\right)\sideset{}{^*}\sum\limits_{\chi(q)}\Bigg|\sum\limits_{\mu X<n\leq X\atop{n=[k^{1/\gamma}]}}\Lambda(n)\chi(n)\Bigg|^2\,,\\
\label{Q1}
&Q_1=\log^4X\,.
\end{align}
First we estimate the sum $\Gamma_3$. Using \eqref{Shapiroformula}, \eqref{Gamma3}, \eqref{Q1} and Lemma \ref{SumLchigama} with $B=2$ we obtain
\begin{align}\label{Gamma3est}
\Gamma_3&\ll X^\gamma(\log X)\sum\limits_{1<q\le Q_1}\frac{1}{\varphi(q)}\sideset{}{^*}\sum\limits_{\chi(q)}\Bigg|\sum\limits_{\mu X<n\leq X\atop{n=[k^{1/\gamma}]}}\Lambda(n)\chi(n)\Bigg|
\ll\frac{X^{2\gamma}}{\log X}\ll X^\gamma Q\log X\,.
\end{align}
Next we consider the sum $\Gamma_4$. Put
\begin{equation}\label{i1}
i_1=\left[\frac{\log\frac{Q}{Q_1}}{\log2}\right]\,.
\end{equation}
By \eqref{Shapiroformula}, \eqref{Gamma4}, \eqref{Q1}, \eqref{i1} and Lemma \ref{largesieve} we derive
\begin{align}\label{Gamma4est}
&\Gamma_4\ll\sum\limits_{1\leq i\le i_1}\sum\limits_{\frac{Q}{2^i}<q\leq\frac{Q}{2^{i-1}}}\frac{1}{\varphi(q)}\left(\log\frac{Q}{q}\right)\sideset{}{^*}\sum\limits_{\chi(q)}
\Bigg|\sum\limits_{\mu X<n\leq X\atop{n=[k^{1/\gamma}]}}\Lambda(n)\chi(n)\Bigg|^2\nonumber\\
&\ll Q^{-1}\sum\limits_{1\leq i\le i_1}2^i\log2^i\sum\limits_{\frac{Q}{2^i}<q\leq\frac{Q}{2^{i-1}}}\frac{q}{\varphi(q)}\sideset{}{^*}\sum\limits_{\chi(q)}
\Bigg|\sum\limits_{\mu X<n\leq X\atop{n=[k^{1/\gamma}]}}\Lambda(n)\chi(n)\Bigg|^2\nonumber\\
&\ll Q^{-1}\sum\limits_{1\leq i\le i_1}2^i i\left(X^\gamma+\frac{Q^2}{4^i}\right)\sum\limits_{n\leq X\atop{n=[k^{1/\gamma}]}}\Lambda^2(n)\nonumber\\    
&\ll Q^{-1}X^\gamma(\log X)\sum\limits_{1\leq i\le i_1}2^i i\left(X^\gamma+\frac{Q^2}{4^i}\right)\nonumber\\     
&\ll X^{2\gamma}Q_1^{-1}\log^3X +X^\gamma Q\log X\nonumber\\     
&\ll X^\gamma Q\log X\,.                     
\end{align}
Taking into account \eqref{Gamma}, \eqref{Gammadecomp}, \eqref{Gamma1est}, \eqref{Gamma2est}, \eqref{Gamma3est} and \eqref{Gamma4est} we establish Theorem \ref{Theorem2}.

\section{Proof of Theorem \ref{Theorem3}}
\indent

From \eqref{deltachi} and the orthogonality of characters we have
\begin{align}\label{orthogonality3}
&\sum\limits_{\mu X<n\leq X\atop{n\equiv a\, (q)\atop{n=[k^{1/\gamma}]}}}\Lambda(n)n^{1-\gamma}e(t n^c)-\frac{\gamma}{\varphi(d)}\int\limits_{\mu X}^Xe(t y^c)\,dy\nonumber\\
&=\sum\limits_{\mu X<n\leq X\atop{n=[k^{1/\gamma}]}}\Lambda(n)n^{1-\gamma}e(t n^c)\frac{1}{\varphi(q)}\sum\limits_{\chi(q)}\chi(n)\overline{\chi}(a)-\frac{\gamma}{\varphi(d)}\int\limits_{\mu X}^Xe(t y^c)\,dy\nonumber\\
&=\frac{1}{\varphi(q)}\sum\limits_{\chi(q)}\Bigg(\overline{\chi}(a)\sum\limits_{\mu X<n\leq X\atop{n=[k^{1/\gamma}]}}\Lambda(n)\chi(n)n^{1-\gamma}e(t n^c)-\delta(\chi)\gamma\int\limits_{\mu X}^Xe(t y^c)\,dy\Bigg)\,.
\end{align}
Now  \eqref{sumchi12} and \eqref{orthogonality3} yield
\begin{align}\label{maxLexpXgamma}
&\sum\limits_{a=1\atop{(a,\,q)=1}}^q\Bigg|\sum\limits_{\mu X<n\leq X\atop{n\equiv a\, (q)\atop{n=[k^{1/\gamma}]}}}\Lambda(n)n^{1-\gamma}e(t n^c)-\frac{\gamma}{\varphi(d)}\int\limits_{\mu X}^Xe(t y^c)\,dy\Bigg|^2\nonumber\\
&=\frac{1}{\varphi^2(q)}\sum\limits_{a=1\atop{(a,\,q)=1}}^q\Bigg|\sum\limits_{\chi(q)}\overline{\chi}(a)
\Bigg(\sum\limits_{\mu X<n\leq X\atop{n=[k^{1/\gamma}]}}\Lambda(n)\chi(n)n^{1-\gamma}e(t n^c)-\delta(\chi)\gamma\int\limits_{\mu X}^Xe(t y^c)\,dy\Bigg)\Bigg|^2\nonumber\\
&=\frac{1}{\varphi(q)}\sum\limits_{\chi(q)}\Bigg|\Bigg(\sum\limits_{\mu X<n\leq X\atop{n=[k^{1/\gamma}]}}\Lambda(n)\chi(n)n^{1-\gamma}e(t n^c)-\delta(\chi)\gamma\int\limits_{\mu X}^Xe(t y^c)\,dy\Bigg)\Bigg|^2\,.
\end{align}
Put
\begin{equation}\label{Omega}
\Omega=\sum\limits_{q\le Q}\sum\limits_{a=1\atop{(a,\,q)=1}}^q
\Bigg|\sum\limits_{\mu X<n\le X\atop{n\equiv a\, ( \textmd{mod}\, q)\atop{n=[k^{1/\gamma}]}}}\Lambda(n)n^{1-\gamma}e(t n^c)-\frac{\gamma}{\varphi(d)}\int\limits_{\mu X}^Xe(t y^c)\,dy\Bigg|^2\,.
\end{equation}
From \eqref{deltachi}, \eqref{maxLexpXgamma} and \eqref{Omega} it follows
\begin{equation}\label{Omegadecomp}
\Omega=\Omega_1+\Omega_2\,,
\end{equation}
where
\begin{align}
\label{Omega1}
&\Omega_1=\sum\limits_{q\le Q}\frac{1}{\varphi(q)}\Bigg|\sum\limits_{\mu X<n\leq X\atop{n=[k^{1/\gamma}]}}\Lambda(n)n^{1-\gamma}e(t n^c)-\gamma\int\limits_{\mu X}^Xe(t y^c)\,dy+\mathcal{O}\Big(X^{1-\gamma}\log^2X\Big)\Bigg|^2\,,
\end{align}
\begin{align}
\label{Omega2}
&\Omega_2=\sum\limits_{q\le Q}\frac{1}{\varphi(q)}\sum\limits_{\chi(q)\atop{\chi\neq\chi_0}}\Bigg|\sum\limits_{\mu X<n\leq X\atop{n=[k^{1/\gamma}]}}\Lambda(n)\chi(n)n^{1-\gamma}e(t n^c)\Bigg|^2\,.
\end{align}
First we consider the sum $\Omega_1$. Using Lemma \ref{SIasympt}, Abel's summation formula and Lemma \ref{Kumchev} we deduce
\begin{align}\label{Omega1est1}
&\sum\limits_{\mu X<n\leq X\atop{n=[k^{1/\gamma}]}}\Lambda(n)n^{1-\gamma}e(t n^c)-\gamma\int\limits_{\mu X}^Xe(t y^c)\,dy\nonumber\\
&=\sum\limits_{\mu X<n\leq X}\Lambda(n)n^{1-\gamma}\big([-n^\gamma]-[-(n+1)^\gamma]\big)e(t n^c)-\gamma\int\limits_{\mu X}^Xe(t y^c)\,dy\nonumber\\
&=\sum\limits_{\mu X<n\leq X}\Lambda(n)n^{1-\gamma}\big((n+1)^\gamma-n^\gamma\big)e(t n^c)-\gamma\int\limits_{\mu X}^Xe(t y^c)\,dy\nonumber\\
&+\sum\limits_{\mu X<n\leq X}\Lambda(n)n^{1-\gamma}\big(\psi(-(n+1)^\gamma)-\psi(-n^\gamma)\big)e(t n^c)\nonumber\\
&=\gamma\sum\limits_{\mu X<n\leq X}\Lambda(n)e(t n^c)-\gamma\int\limits_{\mu X}^Xe(t y^c)\,dy\nonumber\\
&+\sum\limits_{\mu X<n\leq X}\Lambda(n)n^{1-\gamma}\big(\psi(-(n+1)^\gamma)-\psi(-n^\gamma)\big)e(t n^c)+\mathcal{O}(1)\nonumber\\
&\ll\frac{X}{e^{(\log X)^{1/5}}}\,.
\end{align}
Thus from \eqref{Omega1} and \eqref{Omega1est1} we get
\begin{equation}\label{Omega1est2}
\Omega_1\ll\frac{X^2}{e^{2(\log X)^{1/5}}}\sum\limits_{q\le Q}\frac{1}{\varphi(q)}\ll\frac{X^2}{\log^AX}\ll  X^{2-\gamma} Q\log X\,.
\end{equation}
Next we estimate $\Omega_2$. Moving to primitive characters from \eqref{Omega2} we find
\begin{align*}
\Omega_2&\ll\sum\limits_{q\le Q}\frac{1}{\varphi(q)}\sum\limits_{r|q\atop{r>1}}\sideset{}{^*}\sum\limits_{\chi(r)}
\Bigg|\sum\limits_{\mu X<n\leq X\atop{n=[k^{1/\gamma}]}}\Lambda(n)\chi(n)n^{1-\gamma}e(t n^c)\Bigg|^2+X^{2(1-\gamma)}Q\log^4X\nonumber\\
&\ll\sum\limits_{1<r\le Q}\left(\sum\limits_{q\le Q\atop{r|q}}\frac{1}{\varphi(q)}\right)\sideset{}{^*}\sum\limits_{\chi(r)}
\Bigg|\sum\limits_{\mu X<n\leq X\atop{n=[k^{1/\gamma}]}}\Lambda(n)\chi(n)n^{1-\gamma}e(t n^c)\Bigg|^2+X^{2(1-\gamma)}Q\log^4X\nonumber\\
\end{align*}

\begin{align}\label{Omega2est}
&\ll\sum\limits_{1<r\le Q}\frac{1}{\varphi(r)}\left(\log\frac{Q}{r}\right)\sideset{}{^*}\sum\limits_{\chi(r)}
\Bigg|\sum\limits_{\mu X<n\leq X\atop{n=[k^{1/\gamma}]}}\Lambda(n)\chi(n)n^{1-\gamma}e(t n^c)\Bigg|^2+X^{2(1-\gamma)}Q\log^4X\nonumber\\
&=\Omega_3+\Omega_4+X^{2(1-\gamma)}Q\log^4X\,,
\end{align}
where
\begin{align}
\label{Omega3}
&\Omega_3=\sum\limits_{1<q\le Q_2}\frac{1}{\varphi(q)}\left(\log\frac{Q}{q}\right)\sideset{}{^*}\sum\limits_{\chi(q)}\Bigg|\sum\limits_{\mu X<n\leq X\atop{n=[k^{1/\gamma}]}}\Lambda(n)\chi(n)n^{1-\gamma}e(t n^c)\Bigg|^2\,,\\
\label{Omega4}
&\Omega_4=\sum\limits_{Q_2<q\le Q}\frac{1}{\varphi(q)}\left(\log\frac{Q}{q}\right)\sideset{}{^*}\sum\limits_{\chi(q)}\Bigg|\sum\limits_{\mu X<n\leq X\atop{n=[k^{1/\gamma}]}}\Lambda(n)\chi(n)n^{1-\gamma}e(t n^c)\Bigg|^2\,,\\
\label{Q2}
&Q_2=(\log X)^{A+2}\,.
\end{align}
First we consider $\Omega_3$. By \eqref{Shapiroformula}, \eqref{Omega3}, \eqref{Q2}, Abel's summation formula and Lemma \ref{SumLexpchigama} with $B=A+1$ we derive
\begin{align}\label{Omega3est}
\Omega_3&\ll X(\log X)\sum\limits_{1<q\le Q_2}\frac{1}{\varphi(q)}\sideset{}{^*}\sum\limits_{\chi(q)}\Bigg|\sum\limits_{\mu X<n\leq X\atop{n=[k^{1/\gamma}]}}\Lambda(n)\chi(n)n^{1-\gamma}e(t n^c)\Bigg|\nonumber\\
&\ll\frac{X^2}{\log^AX}\ll  X^{2-\gamma} Q\log X\,.
\end{align}
Next we estimate $\Omega_4$. Put
\begin{equation}\label{i2}
i_2=\left[\frac{\log\frac{Q}{Q_2}}{\log2}\right]\,.
\end{equation}
We use \eqref{Omega4}, \eqref{Q2}, \eqref{i2} and Lemma \ref{largesieve} to see that
\begin{align*}
&\Omega_4\ll\sum\limits_{1\leq i\le i_1}\sum\limits_{\frac{Q}{2^i}<q\leq\frac{Q}{2^{i-1}}}\frac{1}{\varphi(q)}\left(\log\frac{Q}{q}\right)\sideset{}{^*}\sum\limits_{\chi(q)}
\Bigg|\sum\limits_{\mu X<n\leq X\atop{n=[k^{1/\gamma}]}}\Lambda(n)\chi(n)n^{1-\gamma}e(t n^c)\Bigg|^2\nonumber\\
&\ll Q^{-1}\sum\limits_{1\leq i\le i_1}2^i\log2^i\sum\limits_{\frac{Q}{2^i}<q\leq\frac{Q}{2^{i-1}}}\frac{q}{\varphi(q)}\sideset{}{^*}\sum\limits_{\chi(q)}
\Bigg|\sum\limits_{\mu X<n\leq X\atop{n=[k^{1/\gamma}]}}\Lambda(n)\chi(n)n^{1-\gamma}e(t n^c)\Bigg|^2\nonumber\\
&\ll Q^{-1}X^{2(1-\gamma)}\sum\limits_{1\leq i\le i_1}2^i i\left(X^\gamma+\frac{Q^2}{4^i}\right)\sum\limits_{n\leq X\atop{n=[k^{1/\gamma}]}}\Lambda^2(n)\nonumber\\  
\end{align*}

\begin{align}\label{Omega4est}  
&\ll Q^{-1}X^{2-\gamma}(\log X)\sum\limits_{1\leq i\le i_1}2^i i\left(X^\gamma+\frac{Q^2}{4^i}\right)\nonumber\\  
&\ll X^2Q_2^{-1}\log^3X +X^{2-\gamma} Q\log X\nonumber\\     
&\ll  X^{2-\gamma} Q\log X\,.                     
\end{align}
Summarizing \eqref{Omega}, \eqref{Omegadecomp}, \eqref{Omega1est2}, \eqref{Omega2est}, \eqref{Omega3est} and \eqref{Omega4est} we establish Theorem \ref{Theorem3}.

\vskip20pt
\footnotesize
\begin{flushleft}
S. I. Dimitrov\\
\quad\\
Faculty of Applied Mathematics and Informatics\\
Technical University of Sofia \\
Blvd. St.Kliment Ohridski 8 \\
Sofia 1756, Bulgaria\\
e-mail: sdimitrov@tu-sofia.bg\\
\end{flushleft}

\begin{flushleft}
Department of Bioinformatics and Mathematical Modelling\\
Institute of Biophysics and Biomedical Engineering\\
Bulgarian Academy of Sciences\\
Acad. G. Bonchev Str. Bl. 105, Sofia 1113, Bulgaria \\
e-mail: xyzstoyan@gmail.com\\
\end{flushleft}

\end{document}